\documentclass[10pt]{amsart}
\usepackage{amsmath,amssymb}

\newcommand\D{{\not\! d}}

\newcommand\liek{{\mathfrak k}}

\newcommand\CC{\mathbb C}

\newcommand\RR{\mathbb R}
\newcommand\ZZ{\mathbb Z}
\newcommand\NN{\mathbb N}

\newcommand\M{{\mathrm{M}}}
\newcommand\GL{{\mathrm{GL}}}
\newcommand\SO{{\mathrm{SO}}}
\newcommand\SU{{\mathrm{SU}}}
\newcommand\U{{\mathrm{U}}}

\newcommand\ind{{\mathrm{ind}}}

\newcommand\End{\operatorname{End}}
\newcommand\Mat{\operatorname{Mat}}

\newcommand\LC{\mathrm{LC}}

\makeatletter 
\newcommand\matc[4]{\left( {#1\@@atop #3}{#2\@@atop #4}\right)}
\newcommand\matr[4]{\left( {\hfill #1\@@atop\hfill #3}{\hfill
#2\@@atop\hfill #4}\right)}
\newcommand\matl[4]{\left( { #1\@@atop #3}{ #2\@@atop\hfill #4}\right)}
\makeatother
\newcommand\widearray[1]{\renewcommand\arraystretch{1.4} \begin{array}{#1}}

\theoremstyle{plain}
\newtheorem{thm}{Theorem}[section]

\newtheorem{prop}[thm]{Proposition}
\newtheorem{cor}[thm]{Corollary}

\theoremstyle{definition}

\title[Pre-sequences of MOP]{Pre-sequences of\\ matrix orthogonal polynomials}
\author{Juan Tirao}
\date{\today}
\thanks{This paper was partially supported by CONICET, PIP 112-200801-01533.}
\address{Juan Tirao\\
CIEM-FaMAF\\ Universidad Nacional de C\'or\-do\-ba\\
C\'or\-do\-ba~5000, Argentina\\
tirao@famaf.unc.edu.ar}



\begin{document}

\begin{abstract} We introduce the notion of a pre-sequence of matrix orthogonal polynomials to mean a sequence $\{F_n\}_{n\ge0}$ of matrix orthogonal functions  with respect to a weight function $W$, satisfying a three term recursion relation and such that  $\det(F_{0})$ is not  zero almost everywhere. By now there is a uniform construction of such sequences from irreducible spherical functions of some fixed $K$-types associated to compact symmetric pairs $(G,K)$ of rank one. Our main result  is  that $\{Q_n=F_nF_0^{-1}\}_{n\ge0}$ is a sequence of matrix orthogonal polynomials with respect to the weight function $F_0WF_0^*$, see Theorem \ref{main}.
\end{abstract}

\maketitle

\begin{section}{introduction}
In the seminal paper \cite{GPT1} a sequence of matrix polynomials $\{F_n\}_{n\ge0}$ was constructed from the irreducible spherical functions of the pair $(G,K)=(\SU(3),\U(2))$, for each  fixed $K$-type. More precisely,  the restrictions of all these functions to the domain of an appropriate local coordinate $x$ of an Iwasawa subgroup $A$ of $G$,  are carefully arranged as the columns of the matrix polynomials $F_n(x)$, $0<x<1$, $n\ge0$.
The irreducible spherical functions of a pair $(G,K)$, where $G$ is a compact group and $K$ a closed subgroup, of a fixed $K$-type $\pi$ are in a one to one correspondence with the unitary irreducible representations of $G$ that contain $\pi$ as a subrepresentation. In our case,
from Schur orthogonality relations of unitary irreducible representations of $\SU(3)$ it follows that $\{F_n\}_{n\ge0}$ is a set of orthogonal polynomials with respect to a weight function $W=W(x)$ supported on the closed interval $[0,1]$. Besides, the irreducible spherical functions are eigenfunctions of the algebra $D(G)^K$ of all left invariant differential operators on $G$ which are also right invariant under $K$. In particular, this implies that $(F_nD)(x)=\Lambda_nF_n(x)$, for all $0<x<1$, where $D$ is a second order right-hand side linear differential operator with matrix valued coefficients.  Moreover, using the explicit branching rules for the tensor product of the fundamental spherical representation and an irreducible unitary representation of $G$, one proves that each sequence $\{F_n\}_{n\ge0}$ satisfies a three term recurrence relation of the form $xF_n=A_nF_{n-1}+B_nF_{n}+C_nF_{n+1}$, where the coefficient matrices are explicitly known. Such sequences $\{F_n\}_{n\ge0}$ are very close to being  sequences of matrix orthogonal polynomials, but they are not, because $F_n$ is not of degree $n$ and its leading coefficient fails to be nonsingular.

The results in \cite{GPT1, GPT2, GPT3} yield for the first time examples of matrix valued classical orthogonal polynomials $\{\Phi_n^{(\alpha,\beta)}\}_{n\ge0}$ of Jacobi type, where the parameters $\alpha, \beta$ take the values $\alpha\in\ZZ_{\ge0}$ and $\beta=1$; classical in the sense that they are eigenfunctions of a fixed second order differential operator.  This was attained from the sequences $\{F_n\}_{n\ge0}$, mentioned above, by defining the matrix valued function $\Phi_n^{\alpha,1}$ by means of $\Phi_n^{\alpha,1}=F_nF_0^{-1}$. This occurred after a useful conversation of F. A. Gr\"unbaum with A. Dur\'an that steered him to this definition, see Section 5 of \cite{GPT3}.  These were the first  examples of such classical sequences of matrix orthogonal polynomials which were sought for by the  experts since the publication \cite{D} of A. Dur\'an.  Then in \cite{G} F. A. Gr\"unbaum gave a completely explicit description of an extension of these sequences $\{\Phi_n^{(\alpha,\beta)}\}$ to arbitrary values of the  parameters $\alpha, \beta>-1$, in the case of size two.  The sequences $\{\Phi_n^{(\alpha,\beta)}\}_{n\ge0}$ become classical sequences of matrix orthogonal polynomials. After that, many more examples of these of polynomials were found by analyzing different compact symmetric pairs of rank one, see for example \cite{PT1, PT2, RT, P, Z, M, PT3, HM, PTZ, TZ},  and by some ad-hoc methods as in \cite{DG1, DG2}.

In \cite{RT} through a trial and error process we found a function $\Psi$ used to play the same role as $F_0$ in \cite{GPT2, GPT3, GPT4, G}, in describing the irreducible spherical functions of any $K$-type associated to the symmetric pairs $(\SU(3), \U(2))$ and $(\SU(2,1), \U(2))$. The same function $\Psi$ was used in \cite{PT2, PT3} to build sequences of classical orthogonal polynomials from sequences of matrix
polynomials $\{F_n\}_{n\ge0}$ constructed from irreducible spherical functions of the pairs $(\SU(d+1),\U(d))$. This same strategy was used in  \cite{PTZ, Z, TZ}.

The sequences of matrix polynomials $\{F_n\}_{n\ge0}$ constructed from the irreducible spherical functions of compact symmetric pairs $(G,K)$ of rank one, for some  fixed $K$-types, are sequences of matrix orthogonal functions  with respect to a weight function $W$, satisfying a three term recursion relation and such that  $\det(F_{0})$ is not  zero almost everywhere. Under these conditions we prove that $\{Q_n=F_nF_0^{-1}\}_{n\ge0}$ is a sequence of matrix orthogonal polynomials with respect to the weight function $F_0WF_0^*$, see Theorem \ref{main}. For this reason such sequences are called pre-sequences of matrix orthogonal polynomials.

The simple and fundamental fact expressed in Theorem \ref{main}, which is the core of this paper, was
conceived in January 2012 and immediately  shared with my coauthors.  At the time Maarten van Pruijssen was working on his Ph.D. thesis  where he extends the results obtained by many of us  for  $(\SU(d+1),\U(d))$ and $(\SO(d+1),\SO(d))$,  to other compact Gelfand pairs $(G,K)$ of rank one. But in these new cases he needs to restrict his attention to those  $K$-types $\pi$ such that $\ind_K^G(\pi)$ is a multiplicity free direct sum of irreducible representations of $G$. In this very nice extension he uses our idea in defining the function $\Psi=F_0$, see Theorem 3.5.12 in \cite{M}, also  consult \cite{HM, MR}.

\end{section}

\begin{section}{Pre-sequences of MOP}
By a weight matrix of size $N$ on the real line $W=W(x)$ we mean
an integrable function on an interval $(a,b)$, such that $W(x)$
is a (self-adjoint) positive semidefinite $N\times N$-matrix for all $x\in (a,b)$, which
is positive definite almost everywhere and with finite moments of all orders: for all
$n\in\NN_0$ we have
$$\int_a^b x^n W(x)\,dx\in\Mat_N(\CC).$$

More generally we could assume that $W$ is a Borel measure on the real line
of $N\times N$-matrices, such that: $W(X)\in\M_N(\CC)$ is positive semidefinite for
any Borel set $X$, $W$ has finite moments of any order, and $W$ is nondegenerate, that
is for P in the polynomial ring $\Mat_N(\CC)[x]$
$$(P,P)=\int_\RR P(x)\,dW(x)P(x)^*=0,$$
only when $P=0$. If $M\in\Mat_N(\CC)$, $M^*$ denotes the conjugate transpose of $M$.

Given a  weight matrix $W=W(x)$ of size $N$ on the real line, supported on the closed interval $[a,b]$,
we introduce as in \cite{K1} and \cite{K2} the following Hermitian
sesquilinear form in the linear space $\Mat_N(\CC)[x]$,
$$(P,Q)=\int_a^bP(x)W(x)Q(x)^*\,dx.$$
The following properties hold:  for all
$P,Q,R\in\Mat_N[x]$,  $T\in\Mat_N(\CC)$ and $a,b\in\CC$ we have
\begin{align*}
(aP+bQ,R)&=a(P,R)+b(Q,R), \\
(TP,Q)&=T(P,Q),\\
(P,Q)^*&=(Q,P),\\
(P,P)&\ge0; \quad \text{if} \;(P,P)=0\;\text{then}\; P=0.
 \end{align*}
These properties follow directly from the definition except the
last one (cf. Proposition 2.2 in \cite{GT}). In other
words $\Mat_N(\CC)[x]$ is a left inner product $\Mat_N(\CC)$-module.

By a sequence of matrix orthogonal polynomials $\{Q_n\}_{n\ge0}$ of size $N$ we mean a sequence of matrix polynomials such that $\deg(Q_n)=n$, the leading coefficient of $Q_n$ is nonsingular and $(Q_m,Q_n)=0$ for all $m\ne n$. Then one establishes that there exists a unique sequence of monic orthogonal polynomials $\{P_n\}_{n\ge0}$ and that any sequence $\{Q_n\}_{n\ge0}$ of matrix orthogonal polynomials is of the form $Q_n=M_nP_n$ where $M_n\in\GL_N(\CC)$ is arbitrary for each $n\ge0$ (cf. Corollary 2.5 in  \cite{GT}).

Two matrix weights are said to be equivalent $W'\sim W$, if  $W'=MWM^*$ for some $M\in \GL_N(\CC)$.  In such a case the map $\{Q_n\}_{n\ge0}\mapsto\{MQ_nM^{-1}\}_{n\ge0}$ establishes a bijection between the set of all sequences of matrix orthogonal polynomials with respect to $W$ and the corresponding set with respect to $W'$.

A standard argument, given for instance in \cite{K1,K2} shows
that a sequence of matrix  orthogonal polynomials $\{Q_n\}_{n\ge0}$
satisfies a three term recursion relation
\begin{equation}\label{eq1}
xQ_n(x) = A_nQ_{n-1}(x) + B_nQ_n(x) + C_nQ_{n+1}(x),\quad n \ge 0
\end{equation}
where $A_0 = 0$ and $C_n$ is a nonsingular matrix .

\

We will consider a sequence $\{F_n\}_{n\ge0}$, of matrix valued functions on $(a,b)$ such that
$\det(F_{0})(x)$ is not  zero almost everywhere. Let $W=W(x)$ be a weight matrix on $(a,b)$ except that we not require $W$ to have finite moments of all orders. If besides $F_i(x)W(x)F_j(x)^*$ is an integrable function on $[a,b]$ for all $i,j\ge0$ and
$$\int_a^b F_i(x)W(x)F_j(x)^*\, dx=0 \quad \text{for all} \quad i\ne j,$$
such a sequence will be called a sequence of  matrix orthogonal functions.

\begin{thm}\label{main} Let $\{F_n\}_{n\ge 0}$ be a sequence of matrix orthogonal functions on $(a,b)$ with respect to a weight function $W$ supported in $[a,b]$. If the sequence satisfies a three term recursion relation of the form
$$xF_n=A_nF_{n-1}+B_nF_n+C_nF_{n+1} \quad n\ge0$$
with $A_{0}=0$ and $C_{n}$ nonsingular, then $F_n=Q_nF_0$ where $Q_n$ is a uniquely determined polynomial of degree $n$ with a nonsingular leading term, for all $n\ge0$.  The moments of $W'=F_0WF_0^*$ are all finite and the sequence $\{Q_n\}_{n\ge0}$ is a sequence of matrix orthogonal polynomials with respect to $W'$, satisfying
\begin{equation}\label{threeterm}
xQ_n=A_nQ_{n-1}+B_nQ_n+C_nQ_{n+1} \quad n\ge 0.
\end{equation}
 \end{thm}
 \begin{proof}
 By induction on $n\ge0$ we will prove that $F_n=Q_{n}F_{0}$ where $Q_{n}$ is a polynomial. For $n=0$ we have $F_{0}=Q_0F_{0}$ which is true for $Q_0=I$.  Now we assume that $F_v=Q_{v}F_{0}$ with $Q_{v}$  polynomial for all $0\le v\le n$. Then we have
\begin{equation*}
\begin{split}
xQ_{n}F_0&=xF_n=A_{n}F_{n-1}+B_{n}F_{n}+C_{n}F_{n+1}\\
             &=A_{n}Q_{n-1}F_{0}+B_{n}Q_{n}F_{0}+C_{n}F_{n+1}.
\end{split}
\end{equation*}
Now by hypothesis $C_n$ is nonsingular for all $n\ge0$. Then we have
$$F_{n+1}=C_n^{-1}\big(xQ_{n}-A_{n}Q_{n-1}-B_{n}Q_{n}\big)F_{0}.$$
Thus we can take
\begin{equation}\label{inductivestep}
Q_{n+1}=C_n^{-1}\big(xQ_{n}-A_{n}Q_{n-1}-B_{n}Q_{n}\big),
\end{equation}
completing the proof of the inductive step.
Now from \eqref{inductivestep} it follows by induction on $n\ge0$ that $Q_{n}$ is a polynomial of degree $n$ with a nonsingular  leading coefficient.

If a matrix polynomial $P=P(x)$ satisfies $P(x)F_0(x)=0$ for all $x$, then $P(x)=0$ for almost all $x$, because by hypothesis $F_0(x)$ is not singular for almost all $x\in(a,b)$.  Hence $P=0$. Taking this into account, the three term recursion relation satisfied by $\{Q_n\}_{n\ge0}$ is obtained from the one in the hypothesis, by canceling $F_{0}$ from both side of such equations. Similarly the  uniqueness of the  $Q_{n}$ follows from the same reason.

Moreover, for all $i,j\ge0$ we have
\begin{equation*}
\begin{split}
\int_a^b Q_i(x)W'(x)Q_j(x)^*\, dx&=\int_a^b Q_i(x) F_0(x)W(x)F_0(x)^*Q_j(x)^*\, dx\\
&=\int_a^b F_i(x)W(x)F_j(x)^*\, dx=\delta_{ij}(F_i,F_j).
\end{split}
\end{equation*}
Hence $\{Q_n\}_{n\ge0}$ is a sequence of matrix orthogonal polynomials with respect to $W'$ which implies that all its moments are finite.
The theorem is proved. \qed
\end{proof}

\

It is worth to observe that the hypothesis of integrability of $F_i(x)W(x)F_j(x)^*$ for all $i,j\ge0$ can be replaced by the hypothesis that $F_0(x)W(x)F_0(x)^*$ has all its moments finite.

We feel appropriate to refer to a sequence $\{F_n\}_{n\ge 0}$ of  matrix orthogonal functions that satisfies a three term recursion relation as \eqref{threeterm} as a {\it pre-sequence of matrix orthogonal polynomials}.

\

We come now to the notion of a differential operator with matrix
coefficients acting on matrix valued polynomials, i.e. elements of
$\Mat_N(\CC)[x]$. These operators could be made to act on our functions either
on the left or on the right. One finds a discussion of these two
actions in \cite{D}. The conclusion there is that if one wants to
have matrix weights that do not reduce to scalar weights and
that have matrix polynomials as their eigenfunctions, one should
settle for right-hand-side differential operators. We agree now to
say that $D$ given by
$$D=\sum_{i=0}^s \partial^iF_i(x),\quad\quad \partial=\frac{d}{dx},$$
acts on $Q(x)$ by means of
$$ QD = \sum_{i=0}^s \partial^i(Q)(x)F_i(x).$$

Given a sequence of matrix orthogonal polynomials $\{Q_n\}_{n\ge0}$ with respect to a weight matrix $W=W(x)$, we introduced in \cite{GT}  the algebra $D(W)$ of all right-hand side
differential operators with matrix valued coefficients that have the
polynomials $Q_w$ as their eigenfunctions.  Thus
\begin{equation}\label{D}
\mathcal D(W)=\{D:Q_nD=\Gamma_n(D)Q_n,\; \Gamma_n(D)\in\Mat_N(\CC)\;\text{for
all}\;n\ge0\}.
\end{equation}

\bigskip

The definition of $\mathcal D(W)$ depends only on the weight matrix $W$ and not on the sequence
$\{Q_n\}_{n\ge0}$, because $Q_n=M_nP_n$ for some $M_n\in\GL_N(\CC)$ where $\{P_n\}_{n\ge0}$ is the sequence of monic orthogonal polynomials. besides , if $D\in\mathcal D(W)$, then $D$ is hypergeometric, meaning that the coefficient of the $j$th-order term is a matrix polynomial of degree less or equal to $j$, see Proposition 2.6 of \cite{GT}.

If $W'\sim W$, say $W'=MWM^*$, $M\in\GL_N(\CC)$,  then the map $D\mapsto MDM^{-1}$ establishes an isomorphism  between the algebras $\mathcal D(W)$ and $\mathcal D(W')$. In fact, if $\{Q_n\}_{n\ge0}$ is a sequence of matrix orthogonal polynomials with respect to $W$, then $\{Q'_n=MQ_nM^{-1}\}_{n\ge0}$ is a sequence of matrix orthogonal polynomials with respect to $W'$. Moreover, if $Q_nD=\Gamma_n(D)Q_n$, then
$$Q'_n(MDM^{-1})= (M\Gamma_n(D)M^{-1})Q'_n.$$
Hence $MDM^{-1}\in\mathcal D(W')$ and $\Gamma_n(MDM^{-1})=M\Gamma_n(D)M^{-1}$.

\begin{cor}\label{algebra} Let $\{F_n\}_{n\ge 0}$ be a pre-sequence of matrix orthogonal polynomials such that $F_0$ is a $C^\infty$-function on $(a,b)$.
 If $F_nD=\Lambda_n F_n$ for all $n\ge0$, where $D$ is a right handed linear ordinary matrix differential operator, then $\tilde D=F_{0}DF_{0}^{-1}\in\mathcal D(F_{0}WF_{0}^*)$. In particular $\tilde D$ is hypergeometric. Moreover if $Q_n=F_n F_0^{-1}$, then
$Q_n\tilde D=\Lambda_n Q_n$.
\end{cor}

\begin{cor} Let $\{F_n\}_{n\ge 0}$ be a pre-sequence of matrix orthogonal polynomials such that $F_0$ is a $C^\infty$-function on $(a,b)$.
 If  $D$ is a second order differential operator whose leading coefficient is $x(1-x)$ and  $F_nD=\Lambda_n F_n$ for all $n\ge0$,  then
 \begin{equation}\label{hypergeometric}
\tilde D=F_{0}DF_{0}^{-1}=\partial^2 x(1-x)+\partial(C-xU)-V,
\end{equation}
where $C,U,V$ are constant matrices.
\end{cor}

If a pre-sequence of matrix orthogonal polynomials $\{F_n\}_{n\ge0}$ comes from   the irreducible spherical functions, of a fixed $K$-type $\pi$,
of a compact symmetric pair $(G,K)$ of rank one, then one can determine  very explicitly a second order differential operator $D=\partial^2x(1-x)+\partial A_1(x)+A_0(x)$ such that $F_nD=\Lambda_n F_n$ for all $n\ge0$. In fact, this can be reached by considering the $\pi$-radial part of the Casimir operator of $G$.  For historical reasons  let $\Psi=F_0$. Then $\Psi D\Psi^{-1}=\tilde D$ is equivalent to $\Psi$ being a solution  of the following differential equations,
\begin{equation}\label{Psi'}
2x(1-x)\Psi'(x)+\Psi(x)A_1(x)=(C-xU)\Psi(x),
\end{equation}
\begin{equation}\label{Psi''}
x(1-x)\Psi''(x)+\Psi'(x)A_1(x)+\Psi(x)A_0(x)=-V\Psi(x).
\end{equation}
If one is able to compute $\Psi(x)^{-1}$,  from \eqref{Psi'} and \eqref{Psi''} one can determine the constant coefficients $C, U, V$ of the hypergeometric differential operator $\tilde D$ given in \eqref{hypergeometric}.

If $\Phi$ is an spherical function on $G$ of $K$-type $\pi$, then $\Phi(kgk')=\pi(k)\Phi(g)\pi(k')$, for all $g\in G$ and $k, k'\in K$. In particular if $X\in\liek$, then $[X\Phi](g)=\Phi(g)\dot\pi(X)$, for all $g\in G$. Let $I(\dot\pi)$ denote the kernel of $\dot\pi$ in $D(K)$ and  consider the left ideal $D(G)I(\dot\pi)$ and the quotient algebra $D(\pi)=D(G)^K/(D(G)^K\cap D(G)I(\dot\pi)).$
Then $D(\pi)$ is commutative, by hypothesis, because it can be embedded, using the Lepowsky antihomomorphism, into the commutative algebra  $D(A)\otimes \End_M(V_\pi)$, $M$ being the centralizer of $A$ in $K$.

Let  $D\in D(G)$. The  $\pi$-radial part $R(\pi,D)$ is a differential operator on $A$ that satisfies $R(\pi,D)(\Phi|_A)=D(\Phi)|_A$. By taking an appropriate local coordinate $x$ on $A$ we can consider $R(\pi,D)$ as a differential operator on  $(0,1)$. If we conjugate $R(\pi,D)$ with $\Psi=F_0$ we obtain a differential operator $\tilde D=\Psi D\Psi^{-1}\in D(\Psi W\Psi^*)$, see Corollary \ref{algebra}. The map $D\mapsto\tilde D$ is an injective homomorphism of $D(\pi)$ into $D(\Psi W\Psi^*)$, which is not surjective because $D(\Psi W\Psi^*)$ is not necessarily commutative. An interesting challenge would be to find an extension of such homomorphism to a global algebra containing $D(G)^K$ in such a way to obtain an isomorphism onto $D(\Psi W\Psi^*)$ (cf. \cite{M}).

\end{section}

\begin{section}{An example}
The following example is taken from Theorem 3.2 in \cite{P}. Let $n\in\NN_0$. Let us consider the following sequence of matrix polynomials
\begin{equation*}
F_w(x)=\begin{pmatrix}  {}_3\!F_2\left (\begin{smallmatrix} -w,\; w+n+3,\; 2\\
      3, \; 1 \end{smallmatrix}\,;x \right )&
     {}_2\!F_1\left (\begin{smallmatrix} -w ,\; w+n+3\\ 3\end{smallmatrix}\,
  ;x\right )\\{}_2\!F_1\left (\begin{smallmatrix} -w ,\; w+n+4\\
3\end{smallmatrix}\,
  ;x\right )&
   {}_3\!F_2\left (\begin{smallmatrix} -w-1,\; w+n+3,\; c+1\\
      3, \; c \end{smallmatrix}\,;x\right)
  \end{pmatrix},
\end{equation*}
$w$ is a nonnegative integer and
$c=\frac{(w+1)(w+n+3)}{(w+1)(w+n+3)+n}$. In particular
$$F_0(x)=\begin{pmatrix} 1&1\\1&1-(n+2)x \end{pmatrix}.$$
We observe that $\det(F_0(x))=-(n+2)x$. Hence $F_0(x)$ is singular only at $x=0$.

The rows of $F_w$ correspond to the irreducible spherical functions of type $(n,1)$ associated to the complex projective plane $P_2(\CC)=\SU(3)/\U(2)$ .

From the orthogonality relations of the irreducible spherical functions, it follows that the sequence $\{F_w\}_{w\ge0}$ is orthogonal with respect to the weight matrix
$$W(x)=\begin{pmatrix}x(1-x)^{n+1}&0\\0&x(1-x)^{n}\end{pmatrix},$$
supported in the closed interval $[0,1]$.

The following theorem is obtained from Theorem 3.5 in \cite{P}  by taking $\alpha=n$,
 $\beta=1$ and using that $F_w(x)=P_w(1-x)\psi^t(1-x)$.
\begin{thm}
The sequence $\{F_w (x)\}_{w\ge0}$ satisfies the following
three term recursion relation
\begin{equation*}\label{difference}
(1- x)F_w(x) =A_{w}F_{w-1}(x) + B_wF_w(x) + C_{w}F_{w+1}(x),
\end{equation*}
with
\begin{align*}
A_w& = \begin{pmatrix}
\frac{w(w+n)(w+n+2)}{(w+n+1)(2w+n+2)(2w+n+3)}
& \frac{w}{(w+1)(w+n+1)(2w+n+3)} \\
0 &
\frac{w(w+2)(w+n+1)}{(w+1)(2w+n+3)(2w+n+4)}
\end{pmatrix},\\
C_w&= \begin{pmatrix}
\frac{(w+1)(w+3)(w+n+3)}{(w+2)(2w+n+3)(2w+n+4)}
& 0\\
\frac{w+3}{(w+2)(w+n+3)(2w+n+4)} &
\frac{(w+3)(w+n+2)(w+n+4)}
{(w+n+3)(2w+n+4)(2w+n+5)}
\end{pmatrix},\\
B_w&=\begin{pmatrix}
B_w^{11} & \frac{w+n+3}{(w+2)(w+n+2)(2w+n+3)}\\
 \frac{w+n+1}{(w+1)(w+n+2)(2w+n+4)}& B_w^{22}
\end{pmatrix},
\end{align*}
where
\begin{align*}
  B_w^{11}&=
  \tfrac{(w+1)^2(w+3)}{(w+2)(2w+n+3)(2w+n+4)}
  + \tfrac{1}{(w+1)(w+2)(w+n+1)(w+n+2)} \\
  &\quad +
  \tfrac{(w+n)(w+n+2)^2}{(w+n+1)(2w+n+2)(2w+n+3)},\\
  B_w^{22}&=
  \tfrac{(w+1)(w+3)^2}{(w+2)(2w+n+4)(2w+n+5)}
  + \tfrac {(w+n+1)^2(w+n+3)}{(w+n+2)
  (2w+n+3)(2w+n+4)}.
\end{align*}
\end{thm}

Therefore the sequence $\{F_w\}_{w\ge0}$ is a pre-sequence of matrix orthogonal polynomials. Hence, according to Theorem \ref{main}, $\{Q_w=F_wF_0^{-1}\}_{w\ge0}$ is a sequence of matrix orthogonal polynomials with very nice properties which can be given explicitly, since $F_0^{-1}$ can be computed.

A careful computation leads to
$$Q_w(x)=F_w(x)F_0(x)^{-1}=\begin{pmatrix} Q^w_{11}(x)&Q^w_{12}(x)\\Q^w_{21}(x)&Q^w_{22}(x)
\end{pmatrix},$$
where
\begin{equation*}
\begin{split}
Q^w_{11}(x)&={}_{3}\!F_2\left (\begin{smallmatrix} -w ,\; w+n+3,\; 2\\
      3, \; 1 \end{smallmatrix}\,; x\right )+
     \tfrac{w(w+n+3)}{3(n+2)}\,  {}_{2}\!F_1\left
     (\begin{smallmatrix} -w+1 ,\; w+n+4\\ 4\end{smallmatrix}\,
  ; x\right ),\\
Q^w_{12}(x)&=\tfrac{-w(w+n+3)}{3(n+2)}\,  {}_{2}\!F_1\left
     (\begin{smallmatrix} -w+1 ,\; w+n+4\\ 4\end{smallmatrix}\,
  ; x\right ),\\
Q^w_{21}(u)&={}_{2}\!F_1\left(\begin{smallmatrix} -w ,\; w+n+4\\
  3\end{smallmatrix}\, ; x \right)
   -\tfrac{s_w}{3(n+2)}\,  {}_{3}\!F_2\left (\begin{smallmatrix} -w ,\; w+n+4,\; s_w+1\\
     4,\; s_w \end{smallmatrix}\, ; x \right),\\
Q^w_{22}(x)&=\tfrac{s_w}{(n+2)}\,  {}_{3}\!F_2\left (\begin{smallmatrix} -w ,\; w+n+4,\; s_w+1\\
     4,\; s_w \end{smallmatrix}\, ; x \right),
\end{split}
\end{equation*}
with $s_w=w(w+n+4)+3(n+2)$.
It is worth to observe at this point that $\deg(F_w)=w+1$ and that $\deg(Q_w)=w$. Moreover, the leading coefficients of $F_w$ and of $Q_w$ are, respectively, the following:
$$\LC(F_w)=\begin{pmatrix} 0&0\\0&\tfrac{(-1)^{w+1}(w+n+3)_{w+1}(c+w+1)}{(3)_{w+1}c}\end{pmatrix},$$
$$\LC(Q_w)=\begin{pmatrix} \tfrac{(-1)^{w}(w+n+3)_{w}2}{(2+w)w!}&0\\\tfrac{(-1)^w(w+n+4)_ww(w-3)}{(3)_{w+1}(n+2)}&\tfrac{(-1)^{w}(w+n+4)_{w}(s_w+w)}
{(n+2)(4)_{w}}\end{pmatrix}.$$
In particular we see that the leading coefficients of $F_w$ are singular while those of $Q_w$ are nonsingular, for all $w\ge0$.

The orthogonality matrix weight of $\{Q_w\}_{w\ge0}$ is
$$W'=F_0WF_0^*=x(1-x)^n\begin{pmatrix}2-x&2-(n+3)x\\2-(n+3)x&1-x+(1-(n+2)x)^2\end{pmatrix},$$
supported in the closed interval $[0,1]$. It is always interesting to wonder if a given matrix weight is irreducible or not. To this end we look at the commuting space $\mathcal C$ of $W'$.

The commuting space $\mathcal C$ of $W'$ is the set of all constant matrices such that $TW'(x)=W'(x)T^*$ for all $x\in[0,1]$. If $T=(t_{ij})_{1\le i,j\le2}$ belongs to $\mathcal C$ if and only if the following equations hold for all $x$,
\begin{equation*}
\begin{split}
(1,1)&:\quad t_{11}(2-x)+t_{12}(2-(n+3)x)=(2-x)\bar t_{11}+(2-(n+3)x)\bar t_{12},\\
(1,2)&:\quad t_{11}(2-(n+3)x)+t_{12}(1-x+(1-(n+2)x)^2)\\
&\hskip5.8cm=(2-x) \bar t_{21}+(2-(n+3)x)\bar t_{22},\\
(2,1)&:\quad t_{21}(2-x)+t_{22}(2-(n+3)x)\\
&\hskip3.3cm=(2-(n+3)x)\bar t_{11}+(1-x+(1-(n+2)x)^2)\bar t_{12},\\
(2,2)&:\quad t_{21}(2-(n+3)x)+t_{22}(1-x+(1-(n+2)x)^2)\\
&\hskip3.3cm=(2-(n+3)x)\bar t_{21}+(1-x+(1-(n+2)x)^2)\bar t_{22}.
\end{split}
\end{equation*}
From $(1,2)$ it follows right away that $t_{12}=0$. Then from $(1,1)$ one obtains that $t_{11}=\bar t_{11}$. Similarly, from $(2,2)$ it follows right away that $t_{22}=\bar t_{22}$, and then that $t_{21}=\bar t_{21}$. At this point equation $(1,2)$ has been reduced to
$$t_{11}(2-(n+3)x)=(2-x)t_{21}+(2-(n+3)x)t_{22}.$$
Since this should hold for all $x$ we obtain that $t_{11}=t_{22}$ and that $t_{21}=0$. Therefore $\mathcal C=\RR I$, which is equivalent to $W'$ being irreducible, see Corollary 1.13 of \cite{T2}.

\

Since the irreducible spherical functions are eigenfunctions of the Casimir operator of $\SU(3)$, it follows that (cf. Section 3 of \cite{P}) $F_w D=\Lambda_w F_w$, $w\ge0$, where
$$D=\partial^2 x(1-x)+\partial\begin{pmatrix}2-(n+4)x&0\\0&2-(n+3)x\end{pmatrix}+
\frac1{x}\begin{pmatrix}-1&1-x\\1&-1+x\end{pmatrix}$$
and
$$\Lambda_w=\begin{pmatrix}-w(w+n+3)&0\\0&-w(w+n+4)-n-2\end{pmatrix}.$$

Now we will compute $F_0DF_0^{-1}$.
\begin{equation*}
\begin{split}
F_0&DF_0^{-1}\\
&=\partial^2 x(1-x)+\partial\left(2F'_0F_0^{-1}x(1-x)+F_0\begin{pmatrix}2-(n+4)x&0\\0&2-(n+3)x\end{pmatrix}F_0^{-1}\right)\\
&\hskip.3cm+F'_0\begin{pmatrix}2-(n+4)x&0\\0&2-(n+3)x\end{pmatrix}F_0^{-1}+\frac1{x}F_0\begin{pmatrix}-1&1-x\\1&-1+x\end{pmatrix}F_0^{-1}.
\end{split}
\end{equation*}
We have
$$F_0^{-1}=\frac{-1}{(n+2)x}\begin{pmatrix}1-(n+2)x&-1\\-1&1\end{pmatrix}\quad\text {and}\quad  F'_0F_0^{-1}=
\frac1{x}\begin{pmatrix}0&0\\-1&1\end{pmatrix},$$
\begin{equation*}
\begin{split}
F_0&\begin{pmatrix}2-(n+4)x&0\\0&2-(n+3)x\end{pmatrix}F_0^{-1}\\
&\hskip1,5cm=\frac1{n+2}\begin{pmatrix}
2n+5-(n+2)(n+4)x&-1\\1-(n+2)x&2n+3-(n+2)(n+3)x\end{pmatrix},
\end{split}
\end{equation*}
$$F'_0\begin{pmatrix}2-(n+4)x&0\\0&2-(n+3)x\end{pmatrix}F_0^{-1}=\frac1{x}\begin{pmatrix}
0&0\\-2+(n+3)x&2-(n+3)x\end{pmatrix},$$
$$F_0\begin{pmatrix}-1&1-x\\1&-1+x\end{pmatrix}F_0^{-1}=\begin{pmatrix}
0&0\\2-(n+3)x&-2+x\end{pmatrix}.$$

This finally gives
$$\tilde D=F_0DF_0^{-1}=\partial^2x(1-x)+\partial(C-xU)-V,$$
where
$$C=\frac1{n+2}\begin{pmatrix} 2n+5&-1\\-2n-3&4n+7\end{pmatrix},\; U=\begin{pmatrix} n+4&0\\-1&n+5\end{pmatrix},\; V=\begin{pmatrix} 0&0\\0&n+2\end{pmatrix}.$$

Since $F_w D=\Lambda_w F_w$ and $Q_w=F_wF_0^{-1}$ for all $w\ge0$, it follows that $Q_w \tilde D=\Lambda_w Q_w$ for all $w\ge0$. Let $Q_{j,w}(x)=(Q_{j,1}(x),Q_{j,2}(x))$ be the $j$th-row of the matrix polynomial $Q_w$, for $j=1,2$. Then
\begin{equation}\label{hypereq}
x(1-x)(Q^t_{j,w})''+(C^t-xU^t)(Q^t_{j,w})'-(V^t+\lambda_{j,w})Q^t_{j,w}=0,
\end{equation}
where $\lambda_{1,w}=-w(w+n+3)$ and $\lambda_{2,w}=-w(w+n+4)-n-2$.

The characteristic polynomial of $C^t$ is $p(\lambda)=(\lambda-2)(\lambda-4)$. Since the eigenvalues of $C^t$ are not in the set $\{0,-1,-2,\dots\}$, the analytic solutions at $x=0$ of \eqref{hypereq} are given in terms of the hypergeometric function  (cf. Definition 3 in \cite{T1})
$$Q^t_{j,w}(x)={}_{2}\!H_1\left(\begin{smallmatrix} U^t,\; V^t+\lambda_{j,w}\\ C^t\end{smallmatrix}\, ; x\right )Q^t_{j,w}(0),$$
with $Q_{1,w}(0)=\frac{w}{3(n+2)}(w+n+3,-w-n-3)$, $Q_{2,w}(0)=\frac{w}{3(n+2)}(-w-n-4,w+n+4)$.

\end{section}

\newcommand\bibit[5]{\bibitem 
{#2}#3, {\em #4,\!\! } #5}

\end{document}